\documentclass[11pt]{article}
\usepackage{amsmath,amsthm,amssymb,amsfonts,mathrsfs,dsfont}
\usepackage[hypertex, bookmarks, colorlinks=true, plainpages = false,
citecolor = green, urlcolor = blue, filecolor = blue]{hyperref}
\newtheorem{theorem}{Theorem}[section]
\newtheorem{lemma}[theorem]{Lemma}
\newtheorem{proposition}[theorem]{Proposition}
\newtheorem{corollary}[theorem]{Corollary}
\newtheorem{definition}[theorem]{Definition}

=23cm\headheight=0cm\topskip=0cm\topmargin=0cm
\def\a{\bar {a}}\def\b{\bar {b}}
\def\x{\bar {x}}\def\y{\bar {y}}

\def\Nn{\mathds N}

\def\Rn{\mathds R}

\def\0{\sf 0}

\makeatletter
\def\dotminussym#1#2{%
  \setbox0=\hbox{$\m@th#1-$}%
  \kern.5\wd0%
  \hbox to 0pt{\hss\hbox{$\m@th#1-$}\hss}%
  \raise.8\ht0\hbox to 0pt{\hss$\m@th#1.$\hss}%
  \kern.5\wd0}
\newcommand{\dotminus}{\mathbin{\mathpalette\dotminussym{}}}

\begin{document}

\begin{center}
{\Large\sc The isomorphism theorem for\\ linear fragments of continuous logic}
\bigskip

{\bf Seyed-Mohammad Bagheri}

\vspace{3mm}

{\footnotesize {\footnotesize Department of Pure Mathematics, Faculty of Mathematical Sciences,\\
Tarbiat Modares University, Tehran, Iran, P.O. Box 14115-134}\\
e-mail:} bagheri@modares.ac.ir
\end{center}

\begin{abstract}
The ultraproduct construction is generalized to $p$-ultramean constructions
($1\leqslant p<\infty$) by replacing ultrafilters with finitely additive measures.
These constructions correspond to the linear fragments $\mathscr L^p$ of continuous logic.
A powermean variant of Keisler-Shelah isomorphism theorem is proved for $\mathscr L^p$.
It is then proved that $\mathscr L^p$-sentences (and their approximations) are
exactly those sentences of continuous logic which are preserved by such constructions.
Some other applications are also given.
%It is also proved that continuous sentences preserved by linear elementary equivalence are
%exactly those sentences which are approximated in the Riesz space generated by linear sentences.
\end{abstract}


\bigskip

A substantial part of the expressive power of every logic comes from its
connectives. Adding new connectives (e.g. infinitary conjunctions)
in order to strengthen the expressive power is a usual way to make a new logic.
Reducing the expressive power by eliminating some connectives is an other way
which has been payed attention by some authors. For example, in \cite{Poizat}
(roughly) the connective $\neg$ is abandoned in order to recover
Robinsonian model theory where non-injective homomorphisms are taken into account.
It is natural to expect that the resulting logic satisfy an appropriate form of compactness
theorem necessary for developing model theory of the intended structures.

In recent years, first order logic has been successfully generalized to
continuous logic by using real numbers (in some bounded interval) as truth values
and continuous real functions as connectives (see \cite{BBHU, CK}).
In fact, with some precautions, one may shift to the whole real line
and use $\{+,r\cdot,\wedge, \vee\}$ as a system of connectives.
Other connectives are then approximated by combinations of these
connectives in the light of the Stone-Weierstrass theorem.
This logic is invented to handle metric spaces equipped with
a continuous structure. Notably, in this framework, compact structures
behave like finite models in first order logic, so that they
can be described by continuous expressions up to isomorphism.

Linear continuous logic is the sublogic of continuous logic
obtained by restricting the connectives to addition $+$ and scalar
multiplication $r\cdot$, hence reducing the expressive power considerably.
This reduction leads to the linearization of most basic tools
and technics of continuous logic such as the ultraproduct construction,
compactness theorem, saturation etc (see \cite{Bagheri-Lip}).
The linear variant of the ultraproduct construction is the ultramean
construction obtained by replacing ultrafilters (as two valued measures)
with arbitrary maximal finitely additive measures.
A consequence of this relaxation is that compact structures with at least
two elements have proper elementary extensions in this sublogic.
In particular, they have now non-categorical theories.
Thus, a model theoretic framework for study of such structures is provided.
A more remarkable aspect of this logic is that the type spaces form a compact convex set.
The extreme types then play a crucial role in the study of linear theories.

The goal of the present paper is to prove the linear variant of Keisler-Shelah
isomorphism theorem for linear continuous logic.
For this purpose, we first show that every continuous structure
has a powermean which is $\aleph_1$-saturated in the linear sense,
i.e. it realizes `linear types' over countable sets of parameters.
It turns out that linearly equivalent saturated structures are partially isomorphic,
hence equivalent in the full continuous logic sense.
In this way, the problem is reduced to the isomorphism theorem for continuous logic.
A consequence of the linear isomorphism theorem is that
linear formulas (and their approximations) are exactly those continuous
formulas which are preserved by the ultramean and powermean constructions.

Some important points should be noted here.
First, a more general powermean construction will be defined and
what we show in this paper is that linearly equivalent
models have isomorphic powermeans in this sense.
The main obstacle in proving the stronger variant is to find `maximal'
ultracharges such that corresponding powermean be $\aleph_1$-saturated.
Second, in \cite{Bagheri-Lip} all functions and relations were considered to be Lipschitz.
We relax this condition here and shift to the more general case where functions
are uniformly continuous in the sense explained below.
Finally, this sublogic has a more general form $\mathscr L^p$ where $p\geqslant1$.
All results proved in the paper hold for every $\mathscr L^p$.
However, to simplify the presentation of the paper, we focus on
the case $p=1$ and discuss the general case at the end of the paper.

\section{Full and linear continuous logics}
All metric spaces will be assumed to have bounded diameters.
By a modulus of uniform continuity (modulus for short) is meant an
increasing continuous concave function $\lambda:\Rn^+\rightarrow\Rn^+$ such that $\lambda(0)=0$.
Such $\lambda$ is then subadditive, i.e. %$f(rt)\leqslant rf(t)$ for $0\leqslant r\leqslant1$
$\lambda(s+t)\leqslant\lambda(s)+\lambda(t)$. Let $X,Y$ be metric spaces.
If $\lambda$ is a modulus, a function $f:X\rightarrow Y$
is called \emph{$\lambda$-continuous} if
$$d(f(x),f(y))\leqslant\lambda(d(x,y)) \hspace{14mm} \forall x,y\in X.$$
$f$ is said to be uniformly continuous if for every $\epsilon>0$
there exists $\delta>0$ such that $d(x,y)<\delta$ implies $d(f(x),f(y))<\epsilon$.
It is not hard to see that every uniformly continuous function in this sense
is $\lambda$-continuous for some suitable modulus $\lambda$ (see \cite{Axler}).

A continuous signature is a set $\tau$ consisting of function,
relation and constant symbols such that to each function symbol $F$
is assigned a modulus $\lambda_F$ and to each relation symbol $R$ is assigned
a modulus $\lambda_R$ as well as a bound ${\bf b}_R\geqslant0$.
It is always assumed that $\tau$ contains a distinguished binary relation symbol $d$
(with $\mathbf{b}_d=1$) which corresponds to $=$ in first order logic.
A $\tau$-structure is a metric space $(M,d)$ on which the symbols of $\tau$ are
appropriately interpreted. That is, for $n$-ary function symbol
$F\in \tau$, the function $F^M:M^n\rightarrow M$
is $\lambda_F$-continuous and similarly for $R\in \tau$, the relation $R^M:M^n\rightarrow\Rn$
is $\lambda_R$-continuous with $\|R^M\|_{\infty}\leqslant{\bf b}_R$.
Here, we put the metric $\sum_{i=1}^nd(x_i,y_i)$ on $M^n$.
In particular, we must have that $diam(M)\leqslant1$.
Let $\tau$ be a signature.
The set of $\tau$-formulas is inductively defined as follows:
$$r,\ \ d(t_1,t_2),\ \ R(t_1,...,t_n),\ \ r\phi,\ \ \phi+\psi,
\ \ ,\phi\wedge\psi,\ \ \phi\vee\psi,\ \ \sup_x\phi,\ \ \inf_x\phi$$
where $R\in \tau$ is $n$-ary, $t_1,...t_n$ are $\tau$-terms and $r\in\Rn$.
In the light of Stone-Weierstrass theorem, other kinds of formulas
such as $\phi\times\psi$ are approximated by the above ones,
hence unnecessary in the formalism.
A formula without free variable is called a \emph{sentence}.
Expressions of the form $\phi\leqslant\psi$ are called \emph{conditions}
(resp. closed conditions if $\phi,\psi$ are sentences).
A \emph{theory} is a set of closed conditions.
If $\phi(\x)$ is a formula, $M$ is a structure and $\a\in M$,
the real value $\phi^M(\a)$ is defined by induction on the complexity of $\phi$.
Every map $\phi^M:M^n\rightarrow\Rn$ is then bounded by some $\mathbf b_\phi$
and continuous with respect to some modulus $\lambda_\phi$ depending only to $\phi$.

The logic based on the set of all formulas (stated above)
is called \emph{continuous logic} and is denoted by CL.
This logic is an extension of first order logic and
it is usually presented in an equivalent way by choosing
the interval $[0,1]$ as value space and $\wedge, \dotminus, \frac{x}{2}$
etc as connectives (see \cite{BBHU}).
The passage to $\Rn$ and using the above connectives has the advantage
of leading us to interesting sublogics of CL such as $\mathscr L^1$ defined below.
%Continuous logic satisfies the compactness theorem
%which the cornerstone of model theory. The intended sublogics to be presented below
%satisfy a stronger form of compactness theorem, the linear compactness.

The set of formulas of the logic $\mathscr L^1$ (also called linear formulas)
is inductively defined as follows:
$$r,\ \ d(t_1,t_2),\ \ R(t_1,...,t_n),\ \ r\phi,\ \ \phi+\psi,\ \ \sup_x\phi,\ \ \inf_x\phi.$$
So, nonlinear connectives $\wedge,\vee$ are not used.
The logic $\mathscr L^1$ has clearly a weaker expressive power than CL.
In particular, the notions of elementary equivalence and elementary embedding are weaker.
Let $M,N$ be $\tau$-structures. We write $M\equiv_{CL} N$ if $\sigma^M=\sigma^N$
for every CL sentence $\sigma$ in $\tau$ and write $M\equiv N$ if
$\sigma^M=\sigma^N$ for every linear sentence $\sigma$ in $\tau$.
So, $\equiv_{CL}$ is stronger than $\equiv$.
Similarly, we write $M\preccurlyeq N$ if $\phi^M(\a)=\phi^N(\a)$ for
every $\a\in M$ and linear formula $\phi$ in $\tau$.

A \emph{linear condition} is an expression of the form $\phi\leqslant\psi$ where
$\phi$ and $\psi$ are linear formulas.
The expression $\phi=\psi$ is an abbreviation for $\{\phi\leqslant\psi, \psi\leqslant\phi\}$.
$M$ is model of a closed condition $\phi\leqslant\psi$ if $\phi^M\leqslant\psi^M$.
A set of closed linear conditions is called a \emph{linear theory}.
A linear theory $T$ is \emph{linearly satisfiable} if for every
conditions $\phi_1\leqslant\psi_1,\ ...,\ \phi_n\leqslant\psi_n$ in $T$
and $0\leqslant r_1,...,r_n$, the condition
$\sum_ir_i\phi_i\leqslant\sum_ir_i\psi_i$ is satisfiable.
The linear compactness theorem can be proved by a linear variant of Henkin's method
or by using ultramean (see \cite{Bagheri-Lip, Bagheri-Safari}).
We will give a shorter proof of this theorem by the ultramean method in section \ref{section4}.

\begin{theorem} \label{compactness}
(Linear compactness) Every linearly satisfiable linear theory is satisfiable.
\end{theorem}

In particular, if $\sigma\leqslant0$ and $0\leqslant\sigma$ are
satisfiable, then $\sigma=0$ is satisfiable.
A satisfiable linear theory $T$ is \emph{linearly complete}
if for each sentence $\sigma$, there is a unique $r$ such that
$\phi=r\in T$. Let $T$ be a linearly complete theory.
A linear \emph{$n$-type} for $T$ is a maximal set $p(\x)$
of linear conditions $\phi(\x)\leqslant\psi(\x)$ such that $T\cup p(\x)$ is satisfiable.
If $p$ is a linear $n$-type, then for each $\phi(\x)$ there is a unique real number $\phi^p$
such that $\phi=\phi^p$ belongs to $p$ and the map $\phi\mapsto\phi^p$ is positive linear.
So, thanks to linear compactness, an $n$-type can be redefined as a linear
functional $p:\mathbb D_n(T)\rightarrow\Rn$ such that for each formula
$\phi(\x)$, the condition $\phi=p(\phi)$ is satisfiable in some model of $T$.
Here, $\mathbb D_n(T)\subseteq \mathbf{C}_b(M^n)$ is the subspace of functions
$\phi^M(\x)$ where $\phi$ is a $\tau$-formula, $|\x|=n$ and $M$ is any model of $T$.
Any function in the completion of $\mathbb D_n(T)$ is called a \emph{definable relation}.

The set of $n$-types of $T$ is denoted by $S_n(T)$.
This is exactly the state space of $\mathbb D_n$, i.e. the set of
functionals $f:\mathbb D_n\rightarrow\Rn$ such that $\|f\|=1=f(1)$.
Types over parameters from $A\subseteq M\vDash T$ are defined similarly.
$S_n(A)$ denotes the set of $n$-types over $A$.
By the Banach-Alaoglu theorem, $S_n(T)$ is a compact
convex subset of the unit ball of $(\mathbb D_n(T))^*$ (similarly for $S_n(A)$).
Some interesting properties of $T$ are related to its extreme types,
i.e. the extreme points of $S_n(T)$.
We recall that $p$ is \emph{extreme} if for every $\lambda\in(0,1)$,\
$p=\lambda p_1+(1-\lambda)p_2$ implies that $p_1=p_2=p$.
A type $p(\x)$ is realized by $\a\in M$ if $p(\phi)=\phi^M(\a)$ for every $\phi$.
%which is partially ordered and normed by
%$$0\leqslant\phi(\x) \ \ \ \ \ \Longleftrightarrow \ \ \ \
%0\leqslant\phi^M(\a)\ \ \ \ \forall\a\in M\vDash T$$
%$$\|\phi\|=\sup\{\phi^M(\a)|\ \ \a\in M\vDash T\}.$$

\begin{definition}
\em{ A model $M\vDash T$ is called (linearly) \emph{$\kappa$-saturated}
if every $p\in S_1(A)$ is realized in $M$ whenever $A\subseteq M$ and $|A|<\kappa$.}
\end{definition}

\begin{lemma}
Let $A\subseteq M$ and $p(\x)$ be a type over $A\subseteq M$.
Then $p$ is realized in some $M\preccurlyeq N$.
\end{lemma}
\begin{proof} It is sufficient to show that $ediag(M)\cup p(\x)$ is linearly satisfiable.
Indeed, for each condition $\phi(\a,\b)\geqslant0$ satisfied in $M$, where $\a\in A$ and $\b\in M-A$,
the condition $\sup_{\y}\phi(\a,\y)\geqslant0$ is satisfiable with $p$ by definition.
\end{proof}

Then, by usual chain arguments one shows that for every $M$ and $\kappa$,
there exists $M\preccurlyeq N$ which is $\kappa$-saturated.
From now on, we always work in the framework of $\mathscr L^1$
and omit the adjective `linear'.
Hence, unless otherwise stated, the linear variant of every notion is intended.
We first recall some properties of charges.

\section{Charges}
In this section we collect some facts about charges which are needed
in the proof of the isomorphism theorem.
A \emph{charge space} is a triple $(I,\mathscr A,\mu)$
where $\mathscr A$ is a Boolean algebra of subsets of $I$
and $\mu:\mathscr A\rightarrow\Rn^+$ is a finitely additive measure.
It is a probability charge if $\mu(I)=1$.
In this paper, by a charge we always mean a probability charge.
If $\mathscr A=P(I)$, $\mu$ is called an \emph{ultracharge}.
Tarski's ultrafilter theorem has an ultracharge variant.

\begin{theorem} (\cite{Rao}, Th. 3.3.)
Let $(I,\mathscr A,\mu)$ be a charge space and $X\subseteq I$.
Assume $$\sup\{\mu(A)|\ A\subseteq X, \ A\in\mathscr A\}\leqslant r\leqslant
\inf\{\mu(A)|\ X\subseteq A,\ A\in\mathscr A\}.$$
Then, there exists an extension of $\mu$ to an ultracharge $\bar\mu$
on $I$ such that $\bar\mu(X)=r$.
\end{theorem}

We recall some properties of integration with respect to charge spaces
(see \cite{Aliprantis-Inf, Rao} for more details).
Let $(I,\mathscr A,\mu)$ be a charge space and
$\mathscr A(\Rn)$ be the Boolean algebra of subsets of $\Rn$
generated by the half-intervals $[u,v)$.
A function $f:I\rightarrow\Rn$ is called $(\mathscr A,\mathscr A(\Rn))$-measurable
if $f^{-1}(X)\in\mathscr A$ for every $X\in\mathscr A(\Rn)$.
Bounded $(\mathscr A,\mathscr A(\Rn))$-measurable functions are always
integrable. % (\cite{Aliprantis-Inf} Th. 11.8).
These functions are sufficient for our purposes.
A sequence $f_n$ of integrable real functions on $I$ converges to $f$
hazily (or in prabability) if for every $\epsilon>0$
$$\lim_n\mu\{i\in I:\ |f_n(i)-f(i)|>\epsilon\}=0.$$

\begin{proposition} \label{convergence} (\cite{Rao}, Th. 4.4.20)
Let $f_n$ be a sequence of integrable functions on $I$ such that
$\lim_{m,n\rightarrow \infty}\int |f_n-f_m| d\mu=0$.
Assume $f_n$ converges hazily to $f$. Then $f$ is integrable and
$\lim_n\int |f_n-f|d\mu=0$. In particular, $\lim_n\int f_n=\int f$.
\end{proposition}
\bigskip

Let $(I,\mathscr A ,\mu)$ be a charge space and $f:I\rightarrow J$ a map.
Define a charge $\nu$ on $J$ by setting
$$\hspace{10mm} \nu(X)=\mu(f^{-1}(X))\hspace{16mm}
\mbox{if}\ X\subseteq J\ \mbox{and}\ f^{-1}(X)\in\mathscr A .$$
We write $\nu=f(\mu)$.
If $\mu$, $\nu$ are charges on $I$ and $J$ respectively and $\nu=f(\mu)$,
we write $\nu\leqslant\mu$.
This defines a partial pre-ordering on the class of charges (or ultracharges)
which generalizes the Rudin-Keisler ordering on ultrafilters (see \cite{Comfort-Negrepontis}).
By change of variables formula, if $\nu=f(\mu)$, then for each bounded integrable
$h:J\rightarrow\Rn$ one has that $\int hd\nu=\int h\circ fd\mu$.

Let $J$ be an infinite index set and for each $r\in J$, $(I_r,\mu_r)$ be an ultracharge space.
Let $\mathcal{I}=\prod_{r}I_r$. A subset $\prod_{r\in J}X_r\subseteq \mathcal{I}$
is called a cylinder if $X_r=I_r$ for all except finitely many $r$.
Let $\mathscr C$ be the Boolean algebra generated by cylinders.
Equivalently, $\mathscr C$ is generated by sets of the form $\pi^{-1}(X_r)$
where $\pi_r:\mathcal I\rightarrow I_r$ is the projection map and $X\subseteq I_r$.
Define a charge $\xi$ by first setting
$$\xi(\prod_{r\in J}X_r)=\prod_{r\in J}\mu_r(X_r)$$
and then extending to $\mathscr C$ in the natural way.
We call $\xi$ the \emph{cylinder charge}.
It is clear that $\mu_r\leqslant\xi$ for all $r$.
Extending $\xi$ to an ultracharge $\mu$, we obtain the following.

\begin{lemma} \label{charge product}
For each $r\in J$, let $\mu_r$ be an ultracharge on a set $I_r$.
Then there exists an ultracharge $\mu$ on $\mathcal I$ such that
$\mu_r\leqslant\mu$ for every $r\in J$.
\end{lemma}

The notion of inverse limit of measures is a well-studied notion in the literature.
Here, we deal with a similar case, the inverse limit of ultracharges.
We consider a special case. Let $(J,<)$ be a linearly ordered
set and $\mu_r$ be an ultracharge on $I_r$ for each $r\in J$.
Assume for each $r\leqslant s$ there exists a surjective map $f_{rs}:I_s\rightarrow I_r$
such that $\mu_r=f(\mu_s)$. Also assume that $f_{rr}=id$ and
$f_{rt}=f_{rs}\circ f_{st}$ whenever $r\leqslant s\leqslant t$.
Let $$\mathbf{I}=\{(i_r)\in\mathcal I|\ \ \ f_{rs}(i_s)=i_r\ \ \ \ \forall\ r\leqslant s\}.$$

\begin{lemma} \label{inverse limit}
There exists a charge $(\mathscr A ,\mu)$ on $\mathbf I$
such that $\mu_r\leqslant\mu$ for all $r\in J$.
%and that, if $\mu_r\leqslant\mu'$ for all $r$, one has that $\mu\leqslant\mu'$.
\end{lemma}
\begin{proof}
Let $\pi_r: \mathbf{I}\rightarrow I_r$ be the projection map.
Let $\mathscr A $ be the subalgebra of $P(\mathbf I)$ generated by sets of the form
$\pi_r^{-1}(X)$ where $X\subseteq I_r$.
Then, there exists a charge $\mu$ on $\mathscr A $
such that $\mu(\pi_r^{-1}(X))=\mu_r(X)$ for all $r$ and $X\subseteq I_r$.
Furthermore, $\mu_r\leqslant\mu$ for each $r$.
%$\mu$ is called the inverse limit of $\{(I_r,\mu_r)\}$ and is denoted by $\varprojlim\mu_r$.
\end{proof}

%\begin{proposition}*
%Let $\kappa\geqslant\aleph_0$. Then, for each family $F$ of ultracharges on $\kappa$ with
%$|F|\leqslant2^\kappa$ there exists an ultracharge $\mu$ such that $\wp\leqslant\mu$ for every $\wp\in F$.
%\end{proposition}

Let $(I,\mathscr A ,\mu)$ be a charge space and $(J,\mathscr B,\nu)$ be an ultracharge space.
For $A\subseteq I\times J$ and $r\in J$ let $A_r=\{i:\ (i,r)\in A\}$.
Let $$\mathscr D=\{A\subseteq I\times J|\ \ \ \forall r\  A_r\in\mathscr A \}.$$
$\mathscr D$ is a Boolean algebra of subsets of $I\times J$ and it is $\kappa$-complete
(resp. the whole power set) if $\mathscr A $ is so.
Since $\nu$ is an ultracharge, we may define a probability charge
on $(I\times J,\mathscr D)$ by setting
$$\wp(A)=\int\mu(A_r)d\nu\ \hspace{10mm} \forall A\in\mathscr D.$$
We denote $\wp$ by $\mu\otimes\nu$.
Note that $\mu,\nu\leqslant\mu\otimes\nu$ via the projection maps.
Then, a one sided Fubini theorem holds.

\begin{lemma}\label{Fubini}
For every bounded $(\mathscr D,\mathscr A(\Rn))$-measurable $f:I\times J\rightarrow\Rn$,
$$\int f(i,j)d\wp=\iint f(i,j)d\mu d\nu.$$
\end{lemma}
\begin{proof} By definition, the claim holds for every $\chi_A$ where $A\in\mathscr D$.
So, it holds for simple functions too.
Let $f$ be as above with range contained in the interval $(-u,u)$.
Let $$\hspace{10mm} f_n(i,j)=\sum_{k=-n}^{n}\frac{k}{n}u\cdot\chi_{A_{nk}}(i,j) \hspace{14mm}
\mbox{where}\ \ A_{nk}=f^{-1}[\frac{k}{n}u, \frac{k+1}{n}u)\in\mathscr D.$$
Then, $f_n$ tends to $f$ uniformly.
Also, for each fixed $j$, $f_n(i,j)$ tends to $f(i,j)$ uniformly and
$$\Big|\int f_n(i,j)d\mu-\int f(i,j)d\mu\Big|\leqslant
\int\big|f_n(i,j)-f(i,j)\big|d\mu\leqslant\frac{u}{n}$$
which shows that $\int f_n(i,j)d\mu$ tends to $\int f(i,j)d\mu$ uniformly on $J$.
So, by Proposition \ref{convergence}
$$\int f(i,j)d\wp=\lim_n\int f_nd\wp=\lim_n\iint f_n(i,j)d\mu d\nu
=\iint\lim_n f_n(i,j)d\mu d\nu=\iint fd\mu d\nu.$$
\end{proof}

A consequence of the lemma is that $\mu\otimes(\nu\otimes\wp)=(\mu\otimes\nu)\otimes\wp$
for any ultracharges $\mu,\nu,\wp$.
%It is worth noting that $\otimes$ is closely related to the notion of
%convolution. If $I$ is a group and $\mu,\nu$ are measures on $I$,
%then $\mu*\nu (A)=\mu\otimes\nu (\bar A)$ where $\bar A=\{(x,y)|\ xy\in A\}$.

%So, in fact, one has that $\mu\leqslant\mu*\nu$ % and $\nu\leqslant\mu*\nu$
%via the maps $$G\stackrel{f}{\longrightarrow}G\times G\stackrel{\times}{\longrightarrow}G$$
%where $f(x)=(x,e)$.

\section{Ultramean and powermean} \label{section4}
The linear variant of the ultraproduct construction is the \emph{ultramean construction}.
Let $(M_{i}, d_{i})_{i\in I}$ be a family of $\tau$-structures.
Let $\mu$ be an ultracharge on $I$.
First define a pseudo-metric on $\prod_{i\in I}M_i$ by setting
$$d(a,b)=\int d_{i}(a_i,b_i)d\mu.$$
Clearly, $d(a,b)=0$ is an equivalence relation.
The equivalence class of $a=(a_i)$ is denoted by $[a_i]$.
Let $M$ be the set of these equivalence classes.
Then, $d$ induces a metric on $M$ which is again denoted by $d$.
So, $d([a_i],[b_i])=\int d_{i}(a_i,b_i)d\mu$.
Define a $\tau$-structure on $(M,d)$ as follows:
$$c^M=[c^{M_i}]$$
$$F^M([a_i],...)=[F^{M_i}(a_i,...)]$$
$$R^M([a_i],...)=\int R^{M_i}(a_i,...)d\mu.$$
where $c,F,R\in \tau$.
One verifies that $F^M$ and $R^M$ are well-defined as well as $\lambda_F$-continuous
and $\lambda_R$-continuous respectively.
For example, assume $R$ is unary. For every $a=[a_i]$ and $b=[b_i]$ one has that
$$R^{M_i}(a_i)-R^{M_i}(b_i)\leqslant\lambda_R(d(a_i,b_i))\ \ \ \ \ \ \forall i.$$
So, by integrating and using Jensen's inequality, one has that
$$R^{M}(a)-R^{M}(b)\leqslant\int\lambda_R(d(a_i,b_i))d\mu
\leqslant\lambda_R\Big(\int d(a_i,b_i)d\mu\Big)=\lambda_R(d(a,b)).$$

The structure $M$ is called the ultramean of structures $M_i$ and is denoted by $\prod_\mu M_i$.
Note that an ultrafilter $\mathcal F$ corresponds to the $0-1$ valued ultracharge
$\mu$ where $\mu(A)=1$ if $A\in\mathcal F$ and $\mu(A)=0$ otherwise.
In this case, $\prod_\mu M_i$ is exactly the ultraproduct $\prod_{\mathcal F}M_i$
and by {\L}o\'{s} theorem, for every CL sentence $\sigma$ in $\tau$
one has that $\sigma^M=\lim_{i,\mathcal F}\sigma^{M_i}$.
In the ultracharge case, we have the following variant of {\L}o\'{s} theorem
(see \cite{Bagheri-Lip}).

\begin{theorem} (Ultramean) \label{th1}
For every linear formula $\phi(x_{1}, \ldots, x_{n})$
and $[a^{1}_{i}], \ldots, [a^{n}_{i}]\in M$
$$\phi^{M}([a^{1}_{i}],\ldots, [a^{n}_{i}])=
\int\phi^{M_{i}}(a^{1}_{i},\ldots, a^{n}_{i})d\mu.$$
\end{theorem}

If $I=\{1,2\}$ and $\mu(1)=\varepsilon$, $\mu(2)=1-\varepsilon$ where $\varepsilon\in[0,1]$,
the ultramean of $(M_i)_{i\in I}$ is denoted by $M=\varepsilon M_1+(1-\varepsilon)M_2$.
In this case, for each linear sentence $\sigma$ we have that
$$\sigma^M=\varepsilon\sigma^{M_1}+(1-\varepsilon)\sigma^{M_2}.$$
Again, using this, one shows that if both $\sigma\leqslant0$
and $0\leqslant\sigma$ are satisfiable, then $\sigma=0$ is satisfiable.
We promised in section one to give a proof of the linear compactness theorem.

Let $T$ be a linearly satisfiable theory in $\tau$. It is not hard to see that for every
sentence $\sigma$, one of the theories $T\cup\{0\leqslant\sigma\}$
and $T\cup\{\sigma\leqslant0\}$ is linearly satisfiable.
So, there is a maximal linearly satisfiable theory containing $T$.
Assume $T$ is a maximal. By the ultraproduct method, if every
$\frac{1}{n}\leqslant\sigma$ is satisfiable, then so is $0\leqslant\sigma$.
As a consequence , for every $\sigma$ there is a unique $r$
such that $\sigma=r$ belongs to $T$.
We denote this unique $r$ by $T(\sigma)$.
In this way, $T$ can be regarded as a linear function on the set of sentences
such that $\sigma=r$ is satisfiable whenever $T(\sigma)=r$.
Let $\mathbb D$ the set of equivalence classes of $\tau$-sentences where
$\sigma\equiv\eta$ if $\sigma^M=\eta^M$ for all $M$.
Then, $\mathbb D$ is partially ordered by $\sigma\leqslant\eta$ if
$\sigma^M\leqslant\eta^M$ for all $M$ and also normed by
$$\|\sigma\|=\sup\{\sigma^M| \ \ M\ \textrm{a}\ \tau\textrm{-structure}\}.$$
In fact, any maximal $T$ can be regarded as a positive linear functional on $\mathbb D$
such that $\sigma=r$ is satisfiable whenever $T(\sigma)=r$.
In particular, if $\sigma\equiv\eta$ then $\sigma-\eta=r$ belongs to $T$ for some $r$.
Since this condition has a model, we conclude that $r=0$ and hence $T(\sigma)=T(\eta)$.
Equivalently, maximal linearly satisfiable theories correspond to linear functionals $f$
on $\mathbb D$ such that $\|f\|_\infty=1=f(1)$.
\bigskip

\noindent{\bf Proof of Theorem \ref{compactness}}:\\
There is a set $\{M_i\}_{i\in I}$ containing a model from each
equivalence class of the relation $M\equiv N$.
So, for every $\tau$-structure $M$, there exists a unique $M_i\in I$
such that $M\equiv M_i$. We put the discrete topology on $I$.
So, we may write $\Rn\subseteq\mathbb D\subseteq\mathbf{C}_b(I)$ if we identify
$\sigma\in\mathbb D$ with the map $i\mapsto\sigma^{M_i}$.
Let $T$ be a linearly satisfiable theory which we may assume maximal
with this property. So, $T$ may be regarded as a positive linear functional on $\mathbb D$.
Since $\mathbb D$ is majorizes $\mathbf{C}_b(I)$, by Kantorovich
extension theorem (\cite{Aliprantis-Inf}, Th. 8.32),
$T$ is extended to a positive linear functional $\bar T$ on $\mathbf{C}_b(I)$.
By (a variant of) Riesz representation theorem
(see \cite{Aliprantis-Inf}, Th.14.9) %or \cite{Rao}, Th. 4.7.4)
there exists probability charge $\mu$ on $I$
such that for every $\sigma$ $$T(\sigma)=\bar{T}(\sigma)=\int_I\sigma^{M_i}d\mu.$$
Extend $\mu$ to an ultracharge $\bar\mu$  and let $N=\prod_{\bar\mu}M_i$.
Then, $T(\sigma)=\sigma^N$ for every $\sigma$ and hence $N$ is a model of $T$. $\square$
\vspace{6mm}

The proof can be arranged to show that if $T$ is linearly satisfiable in a class
$\mathcal K$ of $\tau$-structures, then $T$ has a model of the form $\prod_\mu M_i$
where every $M_i$ belongs to $\mathcal K$.
Using this, one proves that a class $\mathcal K$ of structures is axiomatizable
if and only if it is closed under elementary equivalence and ultramean.

If $M_i=N$ for all $i$, the ultramean structure is denoted by $N^\mu$
and is called \emph{ultra-powermean} (or maximal powermean) of $N$.
Note that if $|N|\geqslant2$ and $\mu$ is not an ultrafilter,
then $N^\mu$ is a proper extension of $N$. Since no compact model
has a proper elementary extension in the CL sense, we conclude that
$\equiv$ is strictly weaker than $\equiv_{CL}$.

The main reason one considers maximal charges in the ultramean theorem is
the integrability of formulas.
The maximal powermean construction has some more flexibilities in this respect and
one may obtain similar results if the tuples are suitably chosen.

Let $M$ be a $\tau$-structure and $(I,\mathscr A,\mu)$ a charge space.
A map (or tuple) $a:I\rightarrow M$ is called \emph{measurable} if
$a^{-1}(X)\in\mathscr A$ for every Borel $X\subseteq M$.
Note that, for a continuous $f:M^n\rightarrow\Rn$ and measurable
tuples $a^1,...,a^n$, the map $f(a^1_i,...,a^n_i)$
may not be $(\mathscr A,\mathscr A(\Rn))$-measurable.
It is however measurable if $density(M)\leqslant\kappa$ and
$\mathscr A $ is $\kappa^+$-complete.
In this case, we say that $M$ is $\mathscr A $-\emph{meanable}.
This includes the case $\mathscr A =P(I)$ as well as the case $M$ is separable and
$\mathscr A $ is a $\sigma$-algebra.
In each of the above cases, let $a\sim b$ if $\int d(a_i,b_i)d\mu=0$.
The equivalence class of $a=(a_i)$ is denoted by $[a_i]$.
Let $M^{\mu}$ be the set of equivalence classes of all measurable maps $a:I\rightarrow M$.
$M^\mu$ is a $\tau$-structure in the natural way.
For example, if $F,R\in \tau$ are unary and $a$ is measurable, one defines
$$F^{M^\mu}([a_i])=[F^M(a_i)], \hspace{14mm} R^{M^\mu}([a_i])=\int R^M(a_i)d\mu.$$
Note also that if $\mu$ is an ultracharge, $M^\mu$ coincides with
the maximal powermean defined above.
To prove the powermean theorem for general $M^\mu$, we need a selection theorem.

Let $\mathscr A\subseteq P(I)$ be a Boolean algebra and $M$ a metric space.
A multifunction $G:I\rightarrow M$ is a map which assigns to each $i$
a nonempty $G(i)\subseteq M$. It is closed-valued if $G(i)$ is closed for each $i$.
It is $\mathscr A$-measurable if for every open $U\subseteq M$, the set
$$G^{-1}(U)=\{i\in I|\ G(i)\cap U\neq\emptyset\}$$
is $\mathscr A$-measurable.
A selection for $G$ is a function $g:I\rightarrow M$ such that $g(i)\in G(i)$ for every $i$.

\begin{theorem} (Kuratowski, Ryll-Nardzewski)
Let $M$ be complete separable metric space and $G:I\rightarrow M$
an $\mathscr A$-measurable closed-valued multifunction.
If $\mathscr A$ is countably complete, then $G$ admits a measurable selection.
\end{theorem}

A proof of the above theorem can be found in \cite{Srivastava}. %see also \cite{Aliprantis-Inf} p.600
It is however not hard to see that the same proof works for any complete $M$
if $density(M)\leqslant\kappa$ and $\mathscr A$ is $\kappa^+$-complete.

Assume $M$ is a complete $\tau$-structure such that $density(M)\leqslant\kappa$
and $(I,\mathscr A,\mu)$ is a charge space where $\mathscr A$ is $\kappa^+$-complete.
Let $\lambda$ be a modulus and $u:M^2\rightarrow\Rn$ be a
$\lambda$-continuous function. Let $B_r(y)$ be the open ball of radius $r$ around $y$.
Then, it is easy to verify that
$$\inf_{t\in B_r(y)}u(x,t)=\inf_z\big[u(x,z)+\lambda\big(d(z,B_r(y))\big)\big].$$
Let $a:I\rightarrow M$ be measurable.
Fix $0<\epsilon<1$ and assume the set
$$G(i)=\{t\in M|\ \ u(a_i,t)<\epsilon\}$$
is nonempty. Then, for each $r>0$,
$$d(y,G(i))<r \ \ \Longleftrightarrow\ \ \inf_{t\in B_r(y)}u(a_i,t)<\epsilon$$% G(i)\cap B_r(z)\neq\emptyset.$$
We deduce that for each $y$, the map $i\mapsto d(y,\overline{G(i)})$ is measurable.
Let $D\subseteq M$ be a dense set with $|D|\leqslant\kappa$ and $U\subseteq M$ be open.
For each $y\in D\cap U$ choose $r_y$ such that
$\frac{d(y, U^c)}{2}<r_y< d(y,U^c)$. Then $U=\bigcup_{y\in D\cap U} B_{r_y}(y)$ and
$$\{i|\ \overline{G(i)}\cap U\neq\emptyset\}=
\bigcup_y\{i|\ \overline{G(i)}\cap B_{r_y}(y)\neq\emptyset\}
=\bigcup_y\{i|\ d(y,\overline{G(i)})<r_y\}\in\mathscr A.$$
This shows that the multifunction $i\mapsto \overline{G(i)}$ is measurable.
Applying this for the function $u(x,y)=[\sup_y\phi^M(x,y)]-\phi^M(x,y)$, we conclude by
Kuratowski and Ryll-Nardzewski theorem that there is a measurable $b:I\rightarrow M$ such that
$$\sup_y\phi^M(a_i,y)-\epsilon\leqslant\phi^M(a_i,b_i)\hspace{12mm}\forall i\in I.$$

%Then, for each open $U\subseteq M$, one has that
%$$\big\{i|\ G(i)\cap U\neq\emptyset\big\}=\{i|\ \
%\exists y[(1-u_i(y))\cdot\chi_U(y)\geqslant1-\epsilon]\}$$
%$$=\{i|\ \ \sup_y[(1-u_i(y))\cdot\chi_U(y)]\geqslant1-\epsilon\}.****$$
%So, $G(i)$ is a measurable multifunction.
%We deduce that there exists a measurable $b:I\rightarrow M$ such that
%$$\sup_y\phi^M(a_i,y)-\epsilon\leqslant\phi^M(a_i,b_i)\hspace{14mm} \forall i\in I.$$
%\begin{definition}{\em ({\sc Temporary})
%Let $(I,\mathscr A,\mu)$ be a charge space and $M$ be $\mathscr A$-meanable.
%Call a family $\mathscr S$ of measurable functions $a:I\rightarrow M$ \emph{selective}
%if for every linear $\phi(x,\y)$, tuples $a^1,...,a^n\in\mathscr S$ and $\epsilon>0$,
%there exists $b\in\mathscr S$ such that
%$$\sup_x\phi^M(b_i,a^1_i,...,a^n_i)-\epsilon\leqslant\phi^M(b_i,a^1_i,...,a^n_i)
%\hspace{12mm} \forall i\in I.$$} \end{definition}
%Then we write $\mathscr S(M)=M^{\mathscr S}=\{[a_i]|\ \ a\in\mathscr A\}$.
%Example: If $M$ is $\mathscr A$-meanable, the following families are selective:
%1. The family of measurable $a:I\rightarrow M$ whose range has cardinality $<\kappa$
%where $\kappa\geqslant\aleph_0$.
%2. The family of measurable $a:I\rightarrow M$ whose range is contained in a compact subset of $M$.

\begin{theorem} \label{powermean} (Powermean)
Let $(I,\mathscr A ,\mu)$ be a charge space and $M$ be a complete
$\mathscr A $-meanable structure. Let $N=M^{\mu}$.
Then, for each $\tau$-formula $\phi(\x)$ and $[a^1_i],...,[a^n_i]\in N$
$$\phi^{N}([a^1_i],...,[a^n_i])=\int\phi^M(a^1_i,...,a^n_i)d\mu.$$
\end{theorem}
\begin{proof} Consider the case $N=M^{\mu}$.
Clearly, the claim holds for atomic formulas.
Also, if it holds for $\phi,\psi$, it holds for $r\phi+s\psi$ too.
Assume the claim is proved for $\phi(\x,y)$.
For simplicity assume $|\x|=1$.
Let $[a_i]\in M^{\mu}$ and $0<\epsilon<1$.
As stated above, there is a measurable $b$ such that
$$\sup_y\phi^M(a_i,y)-\epsilon\leqslant\phi^M(a_i,b_i)\hspace{14mm}\forall i\in I.$$
So,
$$\int\sup_y\phi^M(a_i,y)d\mu-\epsilon\leqslant\int\phi^M(a_i,b_i)d\mu=
\phi^{N}([a_i],[b_i])\leqslant\sup_y\phi^{N}([a_i],y)$$
and hence $$\int\sup_y\phi^M(a_i,y)d\mu\leqslant\sup_y\phi^{N}([a_i],y).$$
The inverse inequality is obvious. So, the claim holds for $\phi(x,y)$ too.
\end{proof}

\begin{corollary} \label{diagonal}
The diagonal map $a\mapsto[a]$ is an elementary embedding of $M$ into $M^\mu$.
\end{corollary}

By change of variables theorem, if $\nu=f(\mu)$ under the map
$f:I\rightarrow J$, then the function $[a_i]_\nu\mapsto[a_{f(i)}]_\mu$
is an elementary embedding of $M^\nu$ into $M^{\mu}$.
In other words, if $\nu\leqslant\mu$ then $M^\nu\preccurlyeq M^\mu$.

\begin{proposition}\label{compose}
Let $(I,\mathscr A,\mu)$ be a charge space and $(J,\nu)$ be an ultracharge space.
If $M$ is complete and $\mathscr A$-meanable, then $M^{\mu\otimes\nu}\simeq(M^{\mu})^{\nu}$.
\end{proposition}
\begin{proof}
Let $\wp=\mu\otimes\nu$ and $[a_{ij}]\in M^{\wp}$. By definition, for each fixed $j$,\ \
$a^j=(a_{ij})_{i\in I}$ is a measurable tuple and its class, which we denote by
$[a^j]_{\mu}$, belongs to $M^{\mu}$.
It is not hard to see that the map $a_\mu\mapsto[[a^j]_{\mu}]_{\nu}$ is
a well-defined bijection. We check that it preserves all formulas.
Let $\phi(\x)$ be a formula (assume $|x|=1$). Then by Lemma \ref{Fubini}
$$\phi^{M^{\wp}}(a_\wp)=\int\phi^{M}(a_{ij})d\wp=\int\int\phi^M(a_{ij})d\mu d\nu=
\int\phi^{M^{\mu}}([a^j]_{\mu})d\nu=\phi^{(M^\mu)^\nu}([[a^j]_{\mu}]_{\nu}).$$
\end{proof}

\section{Saturation and isomorphism theorem}
In this section we prove that for each $M$ and $\kappa$, there is charge
$\mu$ such that $M^\mu$ is $\kappa$-saturated.
Combining this with the CL variant of the isomorphism theorem,
we deduce the isomorphism theorem for the logic $\mathscr L^1$.
Recall that for every positive linear function
$\Lambda:\ell^\infty(X)\rightarrow\Rn$ with $f(1)=1$ there exists a unique
ultracharge $\mu$ such that $\Lambda(f)=\int fd\mu$ for every $f$
(see \cite{Rao}, Th. 4.7.4).
For every set $X$ let $\mathbb{U}(X)$ be the set of ultracharges on $X$.
This is a compact convex set.

\begin{lemma}
$\wp\in\mathbb{U}(X)$ is extreme if and only if it is an ultrafilter.
\end{lemma}
\begin{proof}
Assume $\wp$ is extreme and $\wp(Y)=\varepsilon\in(0,1)$ for some $Y\subseteq X$.
For $A\subseteq X$ set
$$\wp_1(A)=\frac{\wp(A\cap Y)}{\varepsilon},\ \ \ \ \ \ \ \wp_2(A)=\frac{\wp(A\cap Y^c)}{1-\varepsilon}.$$\
Then, $\wp=\varepsilon\wp_1+(1-\varepsilon)\wp_2$ and $\wp\neq\wp_1,\wp_2$.
This is a contradiction. Conversely assume $\wp$ is 2-valued and $\wp=\varepsilon\wp_1+(1-\varepsilon)\wp_2$ where $0<\varepsilon<1$.
Let $A\subseteq X$. Then both $\wp(A)=0$ and $\wp(A)=1$ imply that $\wp_1(A)=\wp_2(A)=\wp(A)$.
\end{proof}

In is well-known in CL that if $\mathcal F$ is a countably
incomplete, then $\prod_{\mathcal F} M_i$ is $\aleph_1$-saturated.
The same proof can be used in $\mathscr L^1$ to show that such an ultraproduct is $\aleph_1$-saturated
if only the extreme types are considered. Let us call a structure $M$
\emph{extremally $\aleph_1$-saturated} if for every countable $A\subseteq M$,
every extreme type in $S_n(A)$ is realized in $M$.

\begin{proposition}\label{sat1}
Let $\mathcal F$ be a countably incomplete ultrafilter on $I$ and for each $i\in I$,
$M_i$ be a $\tau$-structure. Then, $M=\prod_{\mathcal F}M_i$ is extremally
$\aleph_1$-saturated.
\end{proposition}
\begin{proof}
If $a^1,a^2,...\in M$, then $(M,a^1,a^2,...)\simeq\prod_{\mathcal F}(M_i,a^1_i,a^2_i,...)$.
So, we may forget the parameters and show directly that every extreme type
$p(\x)$ of $Th(M)$ is realized in $M$. For simplicity assume $|\x|=1$.
Let $$V=\big\{\wp\in\mathbb{U}(M)|\ \ \ p(\phi)=\int\phi^M(x)d\wp\ \ \ \ \forall\phi\big\}.$$
We may consider $p$ as a function defined on the space of functions $\phi^M(x)$.
Then, by Kantorovich extension theorem (\cite{Aliprantis-Inf}, Th. 8.32)
$p$ extends to a positive linear functional $\bar p$ on $\ell^{\infty}(M)$.
Then, $\bar p$ is represented by integration over an ultracharge on $M$
so that $V$ is non-empty.
Moreover, $V$ is a closed face of $\mathbb{U}(M)$.
In particular, suppose $\varepsilon\mu+(1-\varepsilon)\nu=\wp\in V$ where $\varepsilon\in(0,1)$.
Define types $p_\mu$, $p_\nu$ by setting for each $\phi(x)$,\
$p_\mu(\phi)=\int\phi^M d\mu$ and $p_\nu(\phi)=\int\phi^M d\nu$.
Then, $\varepsilon p_\mu+(1-\varepsilon)p_\nu=p$.
We have therefore that $p_\mu=p_\nu=p$ and hence $\mu,\nu\in V$.

Let $\wp$ be an extreme point of $V$.
Then, $\wp$ is an extreme point of $\mathbb{U}(M)$ and hence corresponds to
an ultrafilter, say $\mathcal D$ (not to be confused with the ultrafilter $\mathcal F$ on $I$).
We have therefore that $$\ \ \ \ p(\phi)=\int_M\phi^M(x)d\wp=
\lim_{\mathcal D,x}\phi^M(x) \ \ \ \ \ \ \ \forall\phi.\hspace{12mm} (*)$$
Since the language is countable, $p$ is axiomatized by a countable set of conditions
say $$p(x)\equiv\{0\leqslant\phi_1(x),\ 0\leqslant\phi_2(x),\ \ldots\ \}.$$
Let $$X_n=\big\{i\in I|\ \ \ -\frac{1}{n}<\sup_x\bigwedge_{k=1}^n\phi_k^{M_i}(x)\ \big\}.$$
Since $\mathcal D$ is an ultrafilter, there exists $a\in M$ such that
$-\frac{1}{n}<\phi^M_k(a)$ for $k=1,...,n$.
We have therefore that $X_n\in\mathcal F$.
Let $I_1\supseteq I_2\supseteq\cdots$ be a chain such that
$I_n\in\mathcal F$ and $\bigcap_n I_n=\emptyset$.
Then, $Y_n=I_n\cap X_n\in\mathcal F$.
Also, $Y_n$ is a decreasing and $\bigcap_n Y_n=\emptyset$.
For $i\not\in Y_1$ let $a_i\in M_i$ be arbitrary.
For $i\in Y_1$, take the greatest $n_i$ such that $i\in Y_{n_i}$
and let $a_i\in M_i$ be such that $-\frac{1}{n_i}\leqslant\phi_k^{M_i}(a_i)$
for $k=1,...,n_i$. Let $a=[a_i]$. Then, for every $n$, if $i\in Y_n$,
we have that $n\leqslant n_i$ and hence $-\frac{1}{n}\leqslant\phi_k^{M_i}(a_i)$ for $k=1,...,n$.
We conclude that $0\leqslant\phi_k^M(a)$ for every $k\geqslant1$, i.e. $a$ realizes $p$.
\end{proof}

In particular, every compact model is extremally $\aleph_1$-saturated.
The proposition \ref{sat1} is interesting in its own right.
However, it does not help us to prove the isomorphism theorem.
As stated in its proof, for every $p$ in $S_1(T)$ (or similarly in $S_1(M)$),
there exists an ultracharge $\mu$ on $M$ which represents $p$.
Let $a:M\rightarrow M$ be the identity map.
Then, we have that $p(\phi)=\int\phi^M(x)d\mu=\phi^{M^\mu}(a)$.
So, $p$ is realized in $M^\mu$. To realize all such types simultaneously,
the corresponding ultracharges must be all dominated by a single ultracharge.

\begin{proposition}\label{sat3}
For each $M$, there exists an ultracharge $\mu$ such that $M^\mu$
realizes all types in every $S_n(M)$. % and $|M^\mu|\leqslant 2^{{|M|^{\aleph_0}+|\tau|+\aleph_0}}$.
\end{proposition}
\begin{proof}
Let $M$ be a $\tau$-structure. We only need to realize the types in $S_1(M)$ in some $M^\mu$.
The types in $S_n(M)$ are then automatically realized in it.
By Lemma \ref{charge product}, there exists a set $I$ and an ultracharge $\wp$ on $I$
such that for every ultracharge $\mu$ on $M$ one has that $\mu\leqslant\wp$.
Given $p(x)\in S_1(M)$, there exists an ultracharge $\mu$
on $M$ such that $p(\phi)=\int\phi^M(x)d\mu$ for every $\phi(x)$.
Let $f:I\rightarrow M$ be such that $f(\wp)=\mu$.
Let $a=[a_i]_\wp$ where $a_i=f(i)$.
Then, $$p(\phi)=\int_M\phi^M(x)d\mu=\int_I\phi^M(a_i)d\wp=\phi^{M^\wp}(a).$$
\end{proof}

%Let $\mathscr A_0$ be a family of subsets of a set $\mathscr{I}$ closed under complementation. Let
%- $\Sigma_0(\mathscr A_0)=\Pi_0(\mathscr A_0)=\mathscr A_0$
%- $\Sigma_{\alpha+1}=$ the family of countable unions of members of $\Pi_\alpha$
%- $\Pi_{\alpha+1}=$ the family of countable unions of members of $\Sigma_\alpha$
%- $\Sigma_\alpha=\bigcup_{\beta<\alpha}\Sigma_\beta$ and $\Pi_\alpha=\bigcup_{\beta<\alpha}\Pi_\beta$
%whenever $\alpha$ is limit.

\begin{theorem} \label{saturated}
Let $M$ be complete and density$(M)\leqslant\kappa$. Then there exists a charge space
$(I,\mathscr A,\wp)$ such that $M$ is $\mathscr A $-meanable and $M^{\wp}$ is $\kappa^+$-saturated.
\end{theorem}
\begin{proof}
First assume $\kappa=\aleph_0$. By a repeated use of Proposition \ref{sat3},
we obtain a countable chain
$$M\preccurlyeq M^{\mu_1}\preccurlyeq (M^{\mu_1})^{\mu_2}\preccurlyeq\cdots$$
where $\mu_n$ is an ultracharge on a set $I_n$.
In the light of Proposition \ref{compose}, we may rewrite it as
$$M\preccurlyeq M^{\nu_1}\preccurlyeq M^{\nu_2}\preccurlyeq\cdots$$
where $\nu_n=\mu_1\otimes\cdots\otimes\mu_n$ is an ultracharge on $J_n=I_1\times\cdots\times I_n$
and $M^{\nu_{n+1}}$ realizes types in $S_1(M^{\nu_n})$.
It is clear that $\{(J_n,\nu_n), f_{mn}\}$, where $f_{mn}:J_n\rightarrow J_m$ is the
projection map, is an inverse system of ultracharges.
The inverse limit of this system is a charge space which can be completed
to an ultracharge space, say $(J_\lambda,\nu_\lambda)$. Then, $\nu_n\leqslant\nu_\lambda$ and
$$M\preccurlyeq M^{\nu_1}\preccurlyeq M^{\nu_2}\preccurlyeq\cdots\preccurlyeq M^{\nu_\lambda}.$$
Iterating the argument, we obtain an inverse system
$\{(J_\alpha, \nu_\alpha), f_{\alpha\beta}\}_{\alpha<\beta<\lambda_1}$
of ultracharges and a chain
$$M\preccurlyeq M^{\nu_1}\preccurlyeq\cdots\preccurlyeq
M^{\nu_{\alpha}}\preccurlyeq\cdots\ \hspace{14mm} \alpha\in\lambda_1$$
such that every $M^{\nu_{\alpha+1}}$ realizes types in $S_1(M^{\nu_\alpha})$.

Let $(\mathbf J,\mathscr B,\nu)$ be the inverse limit of
$\{(J_\alpha,\nu_\alpha), f_{\alpha\beta}\}_{\alpha<\beta<\lambda_1}$
given by Lemma \ref{inverse limit}.
Let $\mathscr A=\sigma(\mathscr B )$ and $\wp$ be an extension of $\nu$ to $\mathscr A $.
So, $(\mathbf J,\mathscr A,\wp)$ is a charge space where $\mathscr A$ is $\aleph_1$-complete.
It is clear that $N=\bigcup_{\alpha<\lambda_1} M^{\nu_\alpha}$ is $\aleph_1$-saturated and that
$N\preccurlyeq M^{\wp}$. We will show that $N=M^\wp$.

Recall that the embedding $M^{\nu_\alpha}\preccurlyeq M^{\nu_\beta}$
takes place via the map $[a_i]\mapsto[b_j]$ where $b_j=a_{f_{\beta\alpha}(j)}$ for every $j\in J_\beta$.
In this way, we identify $[b_j]$ with $[a_i]$.
An element of $M^\wp$ is of the form $[a_{\bf r}]$ where ${\bf r}=(r_\gamma)\in\mathbf{J}$
and $f_{\beta\gamma}(r_\gamma)=r_{\beta}$.
We show that for every such $[a_{\bf r}]$ there exists $\alpha<\lambda_1$
such that $a_{\bf r}$ does not depend on $r_\gamma$ when $\alpha\leqslant\gamma$.
In other words, for each ${\bf r},{\bf s}\in\mathbf{J}$, if $r_\gamma=s_\gamma$ for all $\gamma\leqslant\alpha$,
then $a_{\bf r}=a_{\bf s}$. This clearly implies that $[a_{\bf r}]\in M^{\nu_\alpha}$.

Fix a countable base $\{U_k\}_{k\in\lambda}$ for $M$.
For every $\alpha$, let $\mathscr C_\alpha$ be the algebra of subsets of $\mathbf{J}$ consisting
of sets of the form $\pi_\beta^{-1}(X)$ where $\beta\leqslant\alpha$ and $X\subseteq J_\beta$.
So, by regularity of $\lambda_1$,\ \ $\mathscr A=\bigcup_{\alpha<\lambda_1}\sigma(\mathscr C_\alpha)$.
In particular, there exists $\alpha<\lambda_1$ such that every $a^{-1}(U_k)$
belongs to $\sigma(\mathscr C_\alpha)$.
Suppose that $a_{\bf r}\neq a_{\bf s}$.
Then, ${\bf r}\in a^{-1}(U)$ and ${\bf s}\in a^{-1}(V)$ for some disjoint basic open sets $U$, $V$.
This implies that there exists $\beta\leqslant\alpha$ and $X\subseteq J_\beta$
such that ${\bf r}\in\pi_\beta^{-1}(X)$ and ${\bf s}\in\pi_\beta^{-1}(X^c)$.
Since, otherwise, one could prove by induction (on the complexity class of $A$)
that for every $A\in\sigma(\mathscr C_\alpha)$, \ \
${\bf r}\in A$ if and only if ${\bf s}\in A$ which is impossible.
We conclude that if $r_\gamma=s_\gamma$ for all $\gamma\leqslant\alpha$, then $a_{\bf r}=a_{\bf s}$.

For arbitrary $\kappa$, one must use an elementary chain of length $\kappa^+$.
Also, the charge $\wp$ defined above must be an extension of $\nu$ to
a $\kappa^+$-complete subalgebra generated by $\mathscr B $ in order
to make $M$ an $\mathscr A $-meanable model.
\end{proof}

The isomorphism theorem in first order logic (as well as continuous logic)
states that for $\tau$-structures $M$ and $N$, if $M\equiv_{CL} N$, then there
exists an ultrafilter $\mathcal F$ such that $M^{\mathcal F}\simeq N^{\mathcal F}$.
To prove the similar result in $\mathscr L^1$ we need the following.

\begin{proposition} \label{10}
Assume $M\equiv N$ and they are linearly $\aleph_0$-saturated. Then $M\equiv_{CL} N$.
\end{proposition}
\begin{proof}
Let $\mathscr{F}$ be the family of maps $f:\{a_1,...,a_n\}\rightarrow\{b_1,...,b_n\}$
where $\a\in M$, $\b\in N$ and $tp^M(\a)=tp^N(\b)$.
We show that $\mathscr{F}:M\rightarrow N$ is a partial isomorphism.
Obviously, $\emptyset\in\mathscr{F}$.
Let $f(\a)=\b$ where $f\in\mathscr{F}$ and $c\in M$. Let
$$p(x)=\{\phi(\b,x)=r|\ \ \phi^M(\a,c)=r\}.$$
If $\phi^M(\a,c)=r$, one has that $N\vDash\inf_x\phi(\b,x)\leqslant r$.
So, there is $e_1\in N$ such that $\phi^N(\b,e_1)\leqslant r$.
Similarly, there is $e_2$ such that $r\leqslant\phi^N(\b,e_2)$.
So, by linear compactness, $\phi(\b,x)$ and hence $p(x)$ is satisfiable.
Let $e\in N$ realize $p(x)$. Then, $f\cup\{(c,e)\}\in\mathscr{F}$.
This is the forth property and the back property is verified similarly.

Now, we show by induction on the complexity of the CL formula $\phi(\x)$ in $\tau$
that whenever $f\in\mathscr{F}$ and $f(\a)=\b$, one has that $\phi^M(\a)=\phi^N(\b)$.
The atomic and connective cases $+,r., \wedge,\vee$ are obvious.
Assume the claim is proved for $\phi(\x,y)$. Let $\sup_y\phi^M(\a,y)=r$ and $f(\a)=\b$.
Given $\epsilon>0$, take $c\in M$ such that $r-\epsilon<\phi^M(\a,c)$.
Take $e\in N$ such that $f\cup\{(c,e)\}\in\mathscr{F}$.
By induction hypothesis, $\phi^M(\a,c)=\phi^N(\b,e)$.
Since $\epsilon$ is arbitrary, one has that $r\leqslant\sup_y\phi^N(\b,y)$.
Similarly, one has that $\sup_y\phi^N(\b,y)\leqslant\sup_y\phi^M(\a,y)$
and hence they are equal. We conclude that $M\equiv_{CL} N$.
\end{proof}

For every linearly complete theory $T$ let $T^{CL}$ be the common theory
(in the CL sense) of linearly $\aleph_0$-saturated models of $T$.
Proposition \ref{10} states that $T^{CL}$ is complete in the CL sense.

\begin{theorem} \label{isomorphism} (Isomorphism theorem)
If $M,N$ are complete and $M\equiv N$, there are charge spaces
$(I,\mathscr A,\mu)$ $(J,\mathscr B, \nu)$ such $M$ is $\mathscr A$-meanable,
$N$ is $\mathscr B$-meanable and that $M^{\mu}\simeq N^{\nu}$.
\end{theorem}
\begin{proof}
By Theorem \ref{saturated}, there are charges $\wp_1,\wp_2$ such that
$M^{\wp_1}$, $N^{\wp_2}$ are $\aleph_1$-saturated (hence complete).
By Proposition \ref{diagonal}, $M^{\wp_1}\equiv N^{\wp_2}$.
By Proposition \ref{10}, $M^{\wp_1}\equiv_{CL} N^{\wp_2}$.
So, by the CL variant of the isomorphism theorem,
there exists an ultrafilter $\mathcal F$
such that $(M^{\wp_1})^{\mathcal F}\simeq(N^{\wp_2})^{\mathcal F}$.
We conclude by Proposition \ref{compose} that
$M^{\wp_1\otimes\mathcal F}\simeq N^{\wp_2\otimes\mathcal F}$.
\end{proof}

In the proof of proposition \ref{sat3},
if $M\equiv N$, one can find (using Lemma \ref{charge product})
a $\mu$ such that $M^\mu$ and $N^\mu$ realize types over $M$
and $N$ respectively.  As a consequence, it is possible to arrange
in Theorem \ref{isomorphism} to have that $\mu=\nu$.

\section{Approximation}
In this section we show that every CL formula which is preserved
by ultramean and powermean is approximated by linear formulas.
Let $\Gamma$ be a set of CL formulas in the signature $\tau$.
A formula $\phi(\x)$ is approximated by formulas in $\Gamma$
if for each $\epsilon>0$, there is a formula $\theta(\x)$ in $\Gamma$ such that
$$M\vDash|\phi(\a)-\theta(\a)|\leqslant\epsilon\ \ \ \ \ \ \ \ \ \forall M\ \ \forall\a\in M.$$
Recall that $\mathbb D$ is the vector space of $\tau$-sentences where $\sigma,\eta$
identified if $\sigma^M=\eta^M$ for every $M$.
$\mathbb D$ is also partially ordered by $\sigma\leqslant\eta$
if $\sigma^M\leqslant\eta^M$ for every $M$ and normed by
$$\|\sigma\|=\sup_M\sigma^M.$$
So, a linearly complete theory is just a norm one positive linear functional
$T:\mathbb D\rightarrow\Rn$ and $M\vDash T$ means that $T(\sigma)=\sigma^M$
for every $\sigma$.
Let $\mathbb T$ be the set of all linearly complete $\tau$-theories.
So, $\mathbb T\subseteq B_1(\mathbb D^*)$.
Put the weak* topology of $\mathbb D^*$ on $\mathbb T$.

\begin{proposition} \label{compact-convex}
$\mathbb T$ is a compact convex Hausdorff space.
\end{proposition}
\begin{proof}
It is clear that for $T_1,T_2\in\mathbb T$ and $0\leqslant\varepsilon\leqslant1$,
$\varepsilon T_1+(1-\varepsilon)T_2\in \mathbb T$. So, $\mathbb T$ is convex.
For compactness, note that $\mathbb T$ is a closed subset of the unit ball of $\mathbb D^*$
hence compact by Alaoglu's theorem (\cite{Aliprantis-Inf}, Th. 6.21).
\end{proof}
\bigskip

A function $f:\mathbb T\rightarrow\Rn$ is called affine if for every
$T_1,T_2\in\mathbb T$ and $0\leqslant\varepsilon\leqslant1$
$$f(\varepsilon T_1+(1-\varepsilon)T_2)=\varepsilon f(T_1)+(1-\varepsilon)f(T_2).$$
The set of all affine continuous functions on $\mathbb T$ is denoted by $A(\mathbb T)$.
This is a Banach space.

\begin{theorem} \label{affine dense} (\cite{Wickstead} Corollary 1.1.12)
Let $\mathbb T$ be a compact convex subset of a locally convex space $E$.
Any subspace of $A(\mathbb T)$ which contains the constants and separates
the points of $\mathbb T$ is dense in $A(\mathbb T)$.
\end{theorem}

Let $\phi$ be a CL sentence in $\tau$.
We say that $\phi$ is preserved by ultramean if for every ultracharge space
$(I,\mu)$ and $\tau$-structures $M_i$, one has that $\phi^{N}=\int\phi^{M_i}d\mu$
where $N=\prod_\mu M_i$. Similarly, $\phi$ is preserved by powermean if
$\phi^{M^\mu}=\phi^M$ for every charge space $(I,\mathscr A,\mu)$ for which
$M$ is $\mathscr A$-meanable.
Linear sentences are preserved by ultramean and powermean.

\begin{theorem} \label{char1}
If $\phi$ is preserved by ultramean and powermean, it is approximated by linear sentences.
\end{theorem}
\begin{proof}
For each linear sentence $\sigma$ define a function $f_\sigma$ on $\mathbb T$
by setting $$f_\sigma(T)=T(\sigma).$$
Clearly, $f_\sigma$ is affine and continuous.
Let $$X=\{f_\sigma:\ \ \sigma\ \mbox{a\ linear}\ \tau\mbox{-sentence}\}.$$
$X$ is a linear subspace of $A(\mathbb T)$ which contains constant functions.
Moreover, if $T_1\neq T_2$, there is a linear sentence $\sigma$ such that
$T_1(\sigma)\neq T_2(\sigma)$. So, $f_\sigma(T_1)\neq f_\sigma(T_2)$.
This shows that $X$ separates points.
By Theorem \ref{affine dense}, $X$ is dense in $A(\mathbb T)$.

Define similarly $f_\phi(T)=\phi^M$ where $M\vDash T\in\mathbb T$.
By Theorem \ref{isomorphism}, if $M\equiv N$,
for some $\mu,\nu$ one has that $M^{\mu}\simeq N^{\nu}$.
Hence, $$\phi^M=\phi^{M^{\mu}}=\phi^{N^{\nu}}=\phi^N.$$
So, $f_\phi$ is well-defined. Let us show that $f_\phi$ is affine.
Let $\varepsilon\in[0,1]$ and $T_1,T_2\in\mathbb T$.
Let $M_1\vDash T_1$ and $M_2\vDash T_2$.
Then, $M=\varepsilon M_1+(1-\varepsilon)M_2$ is a model of the theory $\varepsilon T_1+(1-\varepsilon)T_2$.
Moreover, since $\phi$ is preserved by ultramean, we have that
$$f_\phi(\varepsilon T_1+(1-\varepsilon)T_2)=\phi^M=\varepsilon\phi^{M_1}+(1-\varepsilon)\phi^{M_2}
=\varepsilon f_\phi(T_1)+(1-\varepsilon)f_\phi(T_2).$$
So, $f_\phi$ is affine.
Note also that $f_\phi$ is continuous, i.e. for each $r$ the sets
$$\{T\in\mathbb T:\ f_\phi(T)\leqslant r\},\ \ \ \ \ \ \ \
\{T\in\mathbb T:\ f_\phi(T)\geqslant r\}$$
are closed. For example, assume $T_k\rightarrow T$
in the weak* topology and $f_\phi(T_k)\leqslant r$ for each $k$.
We show that $f_\phi(T)\leqslant r$.
Take a nonprincipal ultrafilter $\mathcal F$ on $\Nn$.
Let $M_k\vDash T_k$ and $M=\prod_{\mathcal F}M_k$.
Then, we have that $M\vDash T$. As a consequence,
$$f_\phi(T)=\phi^M=\lim_{k,\mathcal F}\phi^{M_k}=\lim_{k,\mathcal F} f_\phi(T_k)\leqslant r.$$

We conclude that $f_\phi\in A(\mathbb T)$.
So, since $X$ is dense, for each $\epsilon>0$
there is a linear sentence $\sigma$ such that for every $T\in\mathbb T$,\ \
$|f_\phi(T)-f_\sigma(T)|\leqslant\epsilon$.
In other words, for every $M$, \  $|\phi^M-\sigma^M|\leqslant\epsilon$.
\end{proof}
\bigskip

A $\tau$-sentence $\sigma$ is preserved by linear elementary equivalence if for every
$M,N$, whenever $M\equiv N$, one has that $\sigma^M=\sigma^N$.
Note that if $\sigma,\eta$ are preserved by $\equiv$ then so does
$\sigma\wedge\eta$ and $\sigma\vee\eta$.
In fact, every sentence in the Riesz space generated by the set of linear sentences is preserved
by $\equiv$. We denote this Riesz space by $\Lambda$.

\begin{proposition}
$\phi$ is preserved by linear elementary equivalence if and only if
it is approximated by the Riesz space $\Lambda$ generated by the set of linear sentences.
\end{proposition}
\begin{proof}
As in the proof of Theorem \ref{char1}, for each $\sigma\in\Lambda$,
define $f_\sigma:\mathbb T\rightarrow\Rn$ by
$$f_\sigma(T)=\sigma^M$$ where $M\vDash T$ is arbitrary.
Let $$X=\{f_\sigma:\ \sigma\in\Lambda\}.$$
Then, $X$ is a sublattice of $\mathbf{C}(\mathbb T)$ which contains $1$ and separates points.
In particular, $-f_\sigma=f_{-\sigma}$, $f_\sigma+f_\eta=f_{\sigma+\eta}$
and $f_\sigma\wedge f_\eta=f_{\sigma\wedge\eta}$.
By the assumption, the function $f_\phi(T)=\phi^M$ for $M\vDash T$ is well-defined.
Since $\phi$ is preserved by ultraproducts, it is shown similar to the proof
of Proposition \ref{char1} that $f_\phi$ is continuous.
So, by the lattice version of Stone-Weierstrass theorem (see \cite{Aliprantis-Inf}
Th. 9.12), $f_\phi$ is approximated by elements of $\Lambda$.
\end{proof}

\section{$\mathscr L^p$-logics}
There are interesting CL theories which are linear in some sense but
not expressible in the framework of $\mathscr L^1$. For example,
(see \cite{Bacak} for definitions)
a complete metric space is a length spaces if
$$\forall xy\ \forall\epsilon>0\ \ \exists z\ \ \ \ \ \ \ \ \ \
d(x,z)^2+d(y,z)^2\leqslant\frac{1}{2}d(x,y)^2+\epsilon.$$
Equivalently,
$$\sup_{xy}\inf_y \big[d(x,z)^2+d(y,z)^2-\frac{1}{2}d(x,y)^2\big]\leqslant0.$$
A complete metric space is a Hadamard space if
$$\forall xy\ \forall\epsilon>0\ \ \exists m\forall z \ \ \ \ \ \ \ \ \ \
d(z,m)^2+\frac{d(x,y)^2}{4}\leqslant\frac{d(z,x)^2+d(z,y)^2}{2}+\epsilon.$$
Hilbert spaces have similar axioms (the parallelogram law).
Also, the theory of abstract $L^p$-spaces is stated by
$$\|x\wedge y\|^p\leqslant\|x\|^p+\|y\|^p\leqslant\|x+y\|^p.$$
Such theories are generally formalizable in the logics $\mathscr L^p$
defined below.

Let $\tau$ be a signature and $1\leqslant p<\infty$.
The set of formulas of $\mathscr L^p$ is inductively defined as follows:
$$r,\ \ d(t_1,t_2)^p,\ \ R(t_1,...,t_n),\ \ r\phi,\ \ \phi+\psi,\ \ \sup_x\phi,\ \ \inf_x\phi.$$
So, exceptionally, the formula $d(t_1,t_2)$ in $\mathscr L^1(\tau)$ is replaced
with $d(t_1,t_2)^p$.
If $M$ is a $\tau$-structure, $M^n$ is equipped with the metric
$$d_n(\x,\y)=\Big(\sum_{k=1}^n d(x_k,y_k)^p\Big)^{\frac{1}{p}}.$$
As before, for every $F,R\in \tau$, we require that $F^M$ to be
$\lambda_F$-continuous and $R^M$ to be $\lambda_R$-continuous
(as well as being bounded by $\mathbf b_R$).
Then, for every formula $\phi$ there exists a modulus $\lambda_\phi$
such that $\phi^M$ is $\lambda_\phi$-continuous.
Note that we can not require to have that
$|\phi^M(\x)-\phi^M(\y)|\leqslant\lambda_\phi (d(\x,\y)^p)$
since this is highly restrictive for $p>1$.

Logical notions as condition, theory, elementary equivalence etc
are defined as before.
The ultramean construction can be carried out similarly.
Let $\mu$ be an ultracharge on $I$ and $(M_i,d_i)$ be a
$\tau$-structure for each $i$. For $a,b\in \prod_i M_i$ set
$$d(a,b)=\|d_i(a_i,b_i)\|_p=\Big(\int d_i(a_i,b_i)^p d\mu\Big)^{\frac{1}{p}}.$$
Then $d$ is a pseudometric on $\prod_iM_i$ and $d(a,b)=0$
defines an equivalence relation on it.
The equivalence class of $(a_i)$ is denoted by $[a_i]$ and the resulting quotient set by
$M=\prod_{\mu}^p M_i$ (or even by the previous notations if there is no risk of confusion).
The metric induced on $M$ is denoted again by $d$.
We also define a $\tau$-structure on $M$ as follows.
For each $c,F\in\tau$ and non-metric $R\in\tau$ (unary for simplicity)
and $(a_i)\in\prod_i M_i$ set
$$c^M=[c^{M_i}],\ \ \ \ \ \ \ \ \ \ \ F^M([a_i])=[F^{M_i}(a_i)]$$
$$R^M([a_i])=\int R^{M_i}(a_i)d\mu.$$
Then, $F^M$ is $\lambda_F$-continuous and that $R^M$ is $\lambda_R$-continuous
and bounded by $\mathbf b_R$. In particular, if $\a=([a^1_i],...,[a^n_i])$
and $\a_i=(a^1_i,...,a^n_i)$, we have that
$$R^{M_i}(\a_i)-R^{M_i}(\b_i)\leqslant\lambda_R\ (d^{M_i}_n(\a_i,\b_i))\hspace{12mm} \forall i.$$
Since $\|f\|_1\leqslant\|f\|_p$ for any integrable $f:I\rightarrow\Rn$,
by integrating and using Jensen's inequality,
$$R^M(\a)-R^M(\b)\leqslant\lambda_R\big(\|d^{M_i}_n(\a_i,\b_i)\|_1\big)\leqslant
\lambda_R\big(\|d^{M_i}_n(\a_i,\b_i)\|_p\big)=\lambda_R(d^M_n(\a,\b)).$$
Hence $M$ is a $\tau$-structure.
The ultramean theorem \ref{th1} as well as the powermean theorem
\ref{powermean} and all results proved in the previous sections
hold similarly in the framework of $\mathscr L^p$.
We just mention some of these results. Below, $\equiv^p$
is the elementary equivalence in the sense of $\mathscr L^p$.

\begin{theorem} (Compactness)
Every linearly satisfiable $\mathscr L^p$-theory in a signature $\tau$ is satisfiable.
\end{theorem}

\begin{theorem} (Axiomatizability)
A class of $\tau$-structures is $\mathscr L^p$-axiomatizable in the signature
$\tau$ if and only if it is closed under $p$-ultramean and $\equiv^p$.
\end{theorem}

\begin{theorem} (Isomorphism)
If $M,N$ are complete and $M\equiv^p N$, then there is a charge space
$(I,\mathscr A,\mu)$ such that $M$, $N$ are $\mathscr A$-meanable and
$M^{\mu}\simeq N^{\nu}$.
\end{theorem}

\begin{theorem} (Characterization)
If $\phi$ is a CL sentence in $\tau$ which is preserved by $p$-ultramean and
$p$-powermean, it is approximated by $\mathscr L^p$-sentences in $\tau$.
\end{theorem}


\begin{thebibliography}{5}
\bibitem{Aliprantis-Inf} C.D. Aliprantis, K.C. Border,
\textit{Infinite dimensional analysis}, third edition, Springer (2006).
\bibitem{Axler} S. Axler, \emph{Uniform continuoity done right}, internet file.
\bibitem{Bacak} M. Ba\v{c}\'{a}k, \emph{Convex analysis and optimization in Hadamard spaces},
De Gruyter Series in Nonlinear Analysis and Applications 22, De Gruyter (2014).
\bibitem{Bagheri-Lip} S.M. Bagheri, Linear model theory for Lipschitz structures,
Arch. Math. Logic 53:897-927 (2014).
\bibitem{Bagheri-Safari} S.M. Bagheri, R. Safari, Preservation theorems in linear continuous logic, Math. Log. Quart. 60, No.3, 168-176 (2014).
\bibitem{BBHU} I. Ben-Yaacov, A. Berenstein, C.W. Henson, A. Usvyatsov,
\textit{Model theory for metric structures}, Model theory with Applications
to Algebra and Analysis, volume 2 (Zoe Chatzidakis, Dugald Macpherson,
Anand Pillay, and Alex Wilkie, eds.), London Math Society Lecture Note Series,
vol. 350, Cambridge University Press, 2008, pp. 315-427.
\bibitem{CK} C.C. Chang and H.J. Keisler, \textit{Continuous model theory}, Princeton University Press (1966).
\bibitem{Comfort-Negrepontis} W.W. Comfort, S. Negrpontis,
\textit{The theory of ultrafilters}, Springer-Verlag (1974).
\bibitem{Conway} J.B. Conway, \textit{A course in functional analysis},
second edition, Springer (1990).
\bibitem{Dudley} R.M. Dudley, \textit{Real analysis and probability},
Cambridge University Press (2004).
\bibitem{Parthasarathy} T. Parthasarathy, \textit{Lectures notes in mathematics 263},
Springer-Verlag (1972).
\bibitem{Phelps} R.R. Phelps, \textit{Lectures on Choquet's theorem}, Springer-Verlag (2001).
\bibitem{Poizat} B. Poizat, A. Yeshkeyev, \textit{Positive Jonsson theories}, Log. Univers. 12 (2018).
\bibitem{Rao} K.P.S. Bhaskara Rao, M. Bhaskara Rao, \textit{Theory of charges}, Academic Press (1983).
\bibitem{Srivastava} S.M. Srivastava, \textit{A course on Borel sets}, Springer-Verlag (1998).
\bibitem{Wickstead} A. W. Wickstead, \emph{Affine functions on compact convex sets}, Internet file.
\end{thebibliography}
\end{document}